\documentclass[11pt,reqno,a4paper]{amsart}
\usepackage[norelsize]{algorithm2e}
\usepackage{snowshovelingsty}

\usepackage{subfiles}

\author{Mohit Bansil}
\address{Department of Mathematics, Michigan State University}
\email{bansilmo@msu.edu}

\author{Jun Kitagawa}
\address{Department of Mathematics, Michigan State University}
\email{kitagawa@math.msu.edu}

\title[]{An optimal transport problem with storage fees}
\subjclass[2010]{49J45, 49K40}
\thanks{JK's research was supported in part by National Science Foundation grant DMS-1700094.}
\begin{document}
\begin{abstract}
We introduce and investigate properties of a variant of the semi-discrete optimal transport problem. In this problem, one is given an absolutely continuous source measure and cost function, along with a finite set which will be the support of the target measure, and a ``storage fee'' function. The goal is then to find a map for which the total transport cost plus the storage fee evaluated on the masses of the pushforward of the source measure is minimized. We prove existence and uniqueness for the problem, derive a dual problem for which strong duality holds, and give a characterization of dual maximizers and primal minimizers. Additionally, we find some stability results for minimizers.
\end{abstract}

\maketitle

\tableofcontents

\section{Introduction}

\subsection{Semi-discrete optimal transport}

We begin by recalling the classical optimal transport problem. Suppose $X$, $Y$ are metric spaces, $c: X\times Y\to\R$ is a measurable \emph{cost function}, and $\mu$, $\nu$ are Borel probability measures on $X$ and $Y$ respectively. Then the \emph{optimal transport problem} or \emph{Monge problem} transporting $\mu$ to $\nu$ is to find a measurable mapping $T: X\to Y$ such that $T_\#\mu=\nu$ (here recall the \emph{pushforward measure} is defined by $T_\#\mu(E)=\mu(T^{-1}(E))$ for any measurable $E\subset Y$), and $T$ satisfies
\begin{align}\label{eqn: monge}
\int_X c(x, T(x)) d\mu = \min_{S_\#\mu=\nu} \int_X c(x, S(x)) d\mu.
\end{align}
If $X$ is a subset of Euclidean space, $\mu$ is absolutely continuous with respect to Lebesgue measure, and $\nu$ is a finite linear combination of delta measures, the above is usually referred to as the \emph{semi-discrete} optimal transport problem.

We will now be interested in the following variant of the semi-discrete optimal transport problem, where we introduce a ``storage fee.'' Fix a finite collection of $N$ points  $Y:=\{y_j\}_{j=1}^N \subset \R^n$ and a function $F: \R^N \to \R$, assume $\mu$ is an absolutely continuous probability measure on $X\subset \R^n$. This variant is to find a pair $(T, \weightvect)$ where $\weightvect=(\weightvect^1,\ldots, \weightvect^N)\in \R^N$ and $T: X\to Y$ is measureable satisfying
\begin{align*}
T_\#\mu = \sum_{j=1}^N \weightvect^j \delta_{y_j}
\end{align*}
such that
\begin{align}\label{eqn: monge ver}
\int_X c(x, T(x)) d\mu + F(\weightvect) = \min_{\tilde \weightvect\in \R^N,\ \tilde{T}_\#\mu = \sum_{j=1}^N \tilde\weightvect^j \delta_{y_j}} \int_X c(x, \tilde{T}(x)) d\mu + F(\tilde\weightvect).
\end{align}
We will consider a relaxation of this problem which we will refer to as the \emph{primal problem} for the remainder of the paper. To define this relaxation, we write $\Pi(\mu, \nu)$ to denote the space of probability measures on $X\times Y$ whose left and right marginals are $\mu$ and $\nu$ respectively. Then, we wish to find a pair $(\gamma, \weightvect)$ where $\weightvect\in \R^N$ and $\gamma\in \Pi(\mu, \sum_{j=1}^N \weightvect^j\delta_{y_j})$, satisfying
\begin{align}\label{eqn: primal problem}
\int_{X\times Y} c(x, y) d\gamma + F(\weightvect) = \min_{\tilde \weightvect\in \R^N,\ \tilde{\gamma}\in \Pi(\mu, \sum_{j=1}^N \tilde \weightvect^j\delta_{y_j})} \int_{X\times Y} c(x, y) d\tilde\gamma + F(\tilde \weightvect).
\end{align}
The above relaxation is the analogue of relaxing the Monge problem \eqref{eqn: monge} in classical optimal transport to the \emph{Kantorovich problem}, which we recall is (fixing Borel probability measures $\mu$ and $\nu$ on any two topological spaces $X$ and $Y$) the problem of finding a measure $\gamma\in \Pi(\mu, \nu)$ satisfying 
\begin{align}\label{eqn: kantorovich}
\int_{X\times Y} c(x, y) d\gamma = \min_{\tilde{\gamma}\in \Pi(\mu, \nu)} \int_{X\times Y} c(x, y) d\tilde\gamma.
\end{align}

Once a minimizing pair in the above primal problem \eqref{eqn: primal problem} is found, it is clear the measure $\gamma$ is a solution in the Kantorovich problem \eqref{eqn: kantorovich} with the choice $\nu=\sum_{j=1}^N \weightvect^j\delta_{y_j}$. Hence under standard conditions on the cost function and $\mu$, it is easily seen that a solution of \eqref{eqn: primal problem} gives rise to a solution of the Monge version of the problem \eqref{eqn: monge ver}. For more details see Subsection \ref{subsection: primal from dual}.

One interpretation of this variant in terms of economics is the following. A manufacturer has a distribution of factories $\mu$, all producing the same product, and is leasing a finite number of warehouses at the locations $y_j$. At the end of each production cycle, the manufacturer must ship all of their product to be stored at some combination of the warehouses. The manufacturer can choose how many units of their product is to be stored at each warehouse, but the leasing company will charge a storage fee given by $F$ based on the capacity used. Additionally, there is a cost associated to the transportation itself given by $c$, and the goal is to minimize the total cost of transport plus storage.
\subsection{Previous results}
The paper \cite{CrippaJimenezPratelli09} deals with the problem presented here in the specific case of cost function given by $c(x, y)=\abs{x-y}^p$, and storage fee function of the form $F(\weightvect)=\sum_{j=1}^N \weightvect^jh_j(\weightvect^j)$ for some functions $h_j$ (note however, the authors mention their results can be extended to more general cost functions satisfying the condition \eqref{def: twist}). We mention our characterization from Subsection \ref{subsec: characterization} matches the characterization of optimizers given in \cite{CrippaJimenezPratelli09}, however, our current result introduces the associated dual problem, and a stability result which are new. Additionally, \cite{CrippaJimenezPratelli09} also analyzes an associated but different variational problem which we do not discuss, our problem is equivalent to what Crippa, Jimenez, and Pratelli refer to as finding an ``optimum,'' while the above reference deals with the additional problem of finding an ``equilibrium.''

There are also a number of results in the literature dealing with the so-called bilevel location problem using the framework of optimal transport: this can be viewed as a two level problem in which there is a ``lower level problem'' equivalent to the problem discussed in this manuscript, followed by a second ``upper level problem'' consisting of minimizing over the locations $\{y_j\}_{j=1}^N$ in the target domain. The paper \cite{MallozziPassarelli17}, analyzes the case when the lower level problem corresponds to our problem with $c(x, y)=\abs{x-y}^2$ in $\R^2\times \R^2$ and $F(\weightvect)=\inner{a}{\weightvect}$ for a fixed vector $a$, and shows existence and uniqueness under certain conditions. The result \cite{CarlierMallozzi18} views the problem in an economic context, their lower level problem is related to a \emph{partial} optimal transport problem with an associated storage fee; note however that their problem is not exactly an optimal transport problem as it arises from the problem of monopolistic pricing, and involves an extra nonlinearity in the definition of Laguerre cells. We emphasize that we do not deal with the ``upper level problem'', while the above two references also analyze that problem as well.

\subsection{Outline}
We begin in Section \ref{eqn: existence of minimizers} by showing existence of minimizers in the variant \eqref{eqn: primal problem}. In Section \ref{section: dual problem}, we derive a  maximization problem that is dual to \eqref{eqn: primal problem} and show strong duality, which is the content of Theorem \ref{thm: strong duality}. In the subsequent Section \ref{section: dual characterization}, we establish some properties of dual maximizers and primal minimizers, which we utilize to show a characterization of optimizers in both (Theorem \ref{thm: conditions for optimizers}). Finally in Section \ref{section: approximation}, we establish some stability results of minimizers of \eqref{eqn: primal problem}, under convergence of the storage fee functions.
\subsection{Notation and conventions}
We will fix some notation and conventions to be used in the remainder of the paper. We continue to fix positive integers $N$ and $n$ and a collection $Y:=\{y_j\}_{j=1}^N\subset \R^n$. We also denote the standard $N$-simplex by
\begin{align*}
 \weightvectset:=\{{\weightvect}\in \R^N\mid \sum_{j=1}^N\weightvect^j=1,\ \weightvect^j\geq 0\},
\end{align*}
 and given a vector ${\weightvect}\in \weightvectset$ we write $\displaystyle\nu_{{\weightvect}}:=\sum_{j=1}^N\weightvect^j\delta_{y_j}$. We reserve the notation $\onevect$ for the vector in $\R^N$ whose components are all $1$.
The space of Borel probability measures on a topological space $X$ will be denoted $\Prob(X)$, while the set of measures on a space $X\times Y$ with left and right marginals equal to measures $\mu\in \Prob(X)$ and $\nu\in \Prob(Y)$ respectively will be written as $\Pi(\mu, \nu)$. Projection from $X\times Y$ to $X$ and $Y$ will be written $\pi_X$ and $\pi_Y$. 

We will also identify any real valued function on $Y$ with a vector in $\R^N$ in the obvious way, and assume that $\spt\mu=X\subset \R^n$ for the remainder of the paper. Also, in order to simplify arguments we will always assume that $X$ is compact and $c: X\times Y\to \R_{\geq 0}$ is at least continuous throughout.

Also given any convex function $f$, we write $\dom(f):=\{x\mid f(x)<\infty\}$ to denote its \emph{effective domain}. The function $F$ will be assumed to be lower semicontinuous on $\R^N$ for Section \ref{eqn: existence of minimizers}, while for Section \ref{section: dual problem} and after we assume $F$ is a proper, closed, convex function, with $\dom(F)\subset \weightvectset$.

\section{Existence of minimizers}\label{eqn: existence of minimizers}
In this section we will prove existence of minimizers for our problem \eqref{eqn: primal problem}.
\begin{defin}
 A collection of $\Gamma\subset \Prob{(X)}$  is said to be \emph{tight} if for any $\epsilon>0$, there exists a compact set $K\subset X$ such that $\mu(K)>1-\epsilon$ for every $\mu\in \Gamma$.
\end{defin}
We recall the following elementary lemma.
\begin{lem}\label{lem: tight}
	Let $\Gamma_\mu$ be the collection of all measures $\gamma$ on $X \times Y$ with left marginal $\mu$, for some fixed $\mu\in \mathcal{P}(X)$. Then $\Gamma_\mu$ is tight.
\end{lem}

\begin{proof}
	Since $X \subset \R^n$, $X$ is separable and so the collection $\{\mu\}$ is tight. Now let $\eps > 0$ be given. Choose $K \subset X$, compact so that $\mu(K) > 1 - \eps$. Note that since $Y$ is finite, $K\times Y$ is also compact. Then for any $\gamma \in \Gamma_\mu$, we find
\begin{align*}
 \gamma(K \times Y)=\mu(K)>1-\eps,
\end{align*}
hence $\Gamma_\mu$ is tight. 
\end{proof}

As a corollary we see that $\Gamma_\mu$ is relatively weakly compact by Prokhorov's Theorem (see \cite[Theorem 5.1]{Billingsley99}). With this compactness in hand, existence of a minimizer follows easily.

\begin{thm}\label{thm: existence}
 There exist minimizers of the primal problem \eqref{eqn: primal problem} if $c$ is continuous and bounded, and $F$ is lower semicontinuous.
\end{thm}

\begin{proof}

Let $\gamma_i \in \Gamma_\mu$ be a minimizing sequence, i.e. $\int_{X\times Y} c d\gamma_i+F({\weightvect_i})$ approaches the minimum value, where $\nu_{\weightvect_i}$ is the right marginal of $\gamma_i$. By the above remark $\Gamma_\mu$ is compact and so there is a subsequence of $\gamma_i$ which we do not relabel, that  converges weakly to some $\gamma \in \Gamma_\mu$. We will show that $\gamma$ is actually a minimizer. Let $\weightvect$ be the vector in $\R^N$ so that $\nu_\weightvect$ is the right marginal of $\gamma$. 

Indeed since $c$ is continuous and bounded, by the definition of weak convergence we have
\[
\lim_{i\to\infty} \int_{X \times Y} c d\gamma_i = \int_{X \times Y} c d\gamma.
\]

Clearly for any $j$, $X \times\{y_j\}$ has empty boundary relative to $X\times Y$. Hence it is a $\gamma$-continuity set and so by the Portmanteau theorem (\cite[Theorem 2.1]{Billingsley99}), 
\begin{align}\label{eqn: right marginal converges}
\limi \weightvect_i^j =\limi \gamma_i(X \times\{y_j\}) = \gamma(X \times\{y_j\})=\weightvect^j.
\end{align}
Since $F$ is lower semicontinuous, we obtain
\[
\int_{X \times Y} c d\gamma + F(\weightvect)\leq \liminf_{i\to \infty}(\int_{X \times Y} c d\gamma_i + F(\weightvect_i))
\]
as desired. 

\end{proof}

\section{The dual problem}\label{section: dual problem}
Our first goal in this section will be to deduce a dual problem associated to our primal problem \eqref{eqn: primal problem}. This problem will be in a similar vein to Kantorovich's dual problem for the classical optimal transport problem as seen for example in \cite[Theorem 1.3]{Villani03}, and the proof will be along similar lines.

As a reminder, for the remainder of the paper, we will assume that $F$ is a proper, closed, convex function, with $\dom(F)\subset \weightvectset$.
\subsection{Strong duality}
In order to state the dual problem, we first recall a basic concept from convex analysis.
\begin{defin}\label{def: legendre transform}
Let $E$ be a Banach space.  If $F: E\to \R\cup \{+\infty\}$ is a proper function (i.e., it is not identically $+\infty$), its \emph{Legendre-Fenchel transform} is the (proper, convex) function $F^*: E^*\to \R\cup \{+\infty\}$ defined for any $y\in E^*$ by
\begin{align*}
 F^*(y):=\sup_{x\in E}(\inner{y}{x}-F(x)),
\end{align*}
where $\inner{y}{x}$ is the duality pairing between elements of $E^*$ and $E$.

If $E=E^*=\R^n$, then $F^*$ is called the \emph{Legendre transform}.
\end{defin}
\begin{rmk}\label{rmk: legendre transform properties}
Since $\weightvectset$ is compact, we see that $F$ is bounded from below everywhere, as any affine function supporting $F$ from below will be bounded on $\dom(F)$. Thus, we see that $F^*$ is actually finitely valued everywhere on $\R^N$ by the definition of Legendre transform.
\end{rmk}
\begin{thm}[Strong duality]\label{thm: strong duality}
There is strong duality, i.e. 
\begin{align}
&\min_{{\weightvect}\in \weightvectset,\ \gamma\in \Pi(\mu, \nu_{\weightvect})}\int_{X\times Y} c d\gamma+F({\weightvect})\notag\\
&=\sup\{-\int_X \varphi \dmu - F^*({\psi})\mid (\varphi, \psi)\in C(X)\times \R^N,\  -\varphi(x) - \psi^j \leq c(x,y_j),\ \forall (x, y_j)\in X\times Y\}\label{eqn: dual problem}.
\end{align}
\end{thm}
\begin{proof}
Let $E = C(X \times Y)$ and note its dual is given by $E^* = \Meas(X \times Y)$, the space of Radon measures on $X\times Y$. Then define $\Theta$,  $\Xi: E \to \R\cup \{+\infty\}$ by 
\begin{align*}
\Theta(u) :=
\begin{cases}
0, & u(x,y) \geq -c(x,y),\quad  \forall(x, y)\in X\times Y\\
+\infty, & \text{ else }
\end{cases}
\end{align*}
and
\begin{align*}
\Xi(u) :=
\begin{cases}
-\int_X \varphi \dmu + F^*({-\psi}), & \exists (\varphi, \psi)\in C(X)\times \R^N\ s.t.\ u(x, y_j) = -\varphi(x) - \psi^j,\  \forall (x, y_j)\in X\times Y,\\
+\infty, & \text{ else},
\end{cases}
\end{align*}
(we will write $u=-\varphi-{\psi}$ as shorthand for the condition in the first case of $\Xi$ above). We can see that $\Xi$ as above is well-defined. Indeed, if $u(x, y_j)=-\phi_1(x)-\psi_1^j=-\phi_2(x)-\psi_2^j$ for all $(x, y_j)\in X\times Y$, we can see there exists some $r\in \R$ such that $\phi_1=-r+\phi_2$, and $\psi_1=\psi_2+r\onevect$. Since $\weightvectset$ is contained in a plane orthogonal to $\onevect$ and $F=\infty$ outside of this plane, a direct verification of the definition implies that for any $\weightvect\in \dom(F)$, $\psi\in \subdiff{F}{\weightvect}$, and $r\in \R$, we must have $\psi+r\onevect\in \subdiff{F}{\weightvect}$ as well. By Remark \ref{rmk: legendre transform properties}, we see $F^*$ is finite everywhere, hence $\subdiff{F^*}{\psi}\neq \emptyset$ for any $\psi$. Thus there exists some $\weightvect$ which must be in $\weightvectset$, such that $\lambda\in \subdiff{F^*}{-\psi_2}$, hence by \cite[Theorem 23.5]{Rockafellar70},
\begin{align}\label{eqn: invariance of dual}
F^*(-\psi_1)&= F^*(-\psi_2-r\onevect)=-\inner{\psi_2}{\weightvect}-r\inner{\onevect}{\weightvect}-F(\weightvect)=-r+F^*(-\psi_2).
\end{align} 
Since $\mu$ is a probability measure, this shows $\Xi$ is well-defined. 
It is immediate to see that $\Theta$ and $\Xi$ are convex, and note for $u \equiv 1$, $\Theta(u)$, $\Xi(u) < +\infty$ and $\Theta$ is continuous at $u$. 

We now compute
\begin{align*}
\Theta^*(-\gamma) 
= \sup_{u\in E} \( -\int_{X\times Y} u \dga - \Theta(u) \)
= \sup_{u \geq -c} \( -\int_{X\times Y} u \dga \)
= \begin{cases}
\int_{X\times Y} c \dga, & \gamma \geq 0 \\
+\infty, & \text{ else }
\end{cases}
\end{align*}
and
\begin{align*}
\Xi^*(\gamma)
&= \sup_{u\in E} \( \int_{X\times Y} u \dga - \Xi(u) \) = \sup_{u=-\varphi - {\psi}} \( \int_{X\times Y}u \dga + \int_X \varphi \dmu - F^*({-\psi}) \) \\
&= \sup_{(\varphi, {\psi})\in C(X)\times \R^N} \( \int_{X\times Y} (-\varphi(x) - {\psi}) \dga + \int_X \varphi \dmu - F^*({-\psi}) \) \\
&= \sup_{(\varphi, {\psi})\in C(X)\times \R^N}\( \int_{X} -\varphi d((\pi_X)_\# \gamma - \mu) -\int_Y {\psi} d(\pi_Y)_\#\gamma - F^*({-\psi}) \) \\
&= \begin{cases}
\displaystyle\sup_{{\psi}} \( \inner{{-\psi}}{{\weightvect}} - F^*({-\psi})\), & (\pi_X)_\# \gamma = \mu,\ (\pi_Y)_\# \gamma = \nu_\weightvect \\
+\infty, & \text{ else }
\end{cases} \\
&= \begin{cases}
F^{**}({\weightvect}), & (\pi_X)_\# \gamma = \mu,\ (\pi_Y)_\# \gamma = \nu_\weightvect \\
+\infty, & \text{ else},
\end{cases}\\
&= \begin{cases}
F({\weightvect}), & (\pi_X)_\# \gamma = \mu,\ (\pi_Y)_\# \gamma = \nu_\weightvect \\
+\infty, & \text{ else},
\end{cases}
\end{align*}
where the last equality above is by convexity of $F$.

Next we find (where by an abuse of notation we will write $-\varphi-\psi\leq c$ to denote $-\varphi(x)-\psi^j\leq c(x, y_j)$ for all $x\in X$ and $1\leq j\leq N$)
\begin{align*}
\inf_{z \in E} (\Theta(z) + \Xi(z)) 
&= \inf_{\varphi + \psi \leq c} \(-\int_X \varphi \dmu + F^*({-\psi})\) \\
&= -\sup_{\varphi + \psi \leq c}\( \int_X \varphi \dmu - F^*({-\psi})\)\\
&= -\sup_{-\varphi - \psi \leq c} \(\int_X (-\varphi) \dmu - F^*({\psi})\) \\
\end{align*}

Hence by the Fenchel-Rockafellar theorem (see \cite[Theorem 1.9]{Villani03}) we have
\begin{align*}
\sup_{-\varphi - \psi \leq c}\( -\int_X \varphi \dmu - F^*({\psi})\)&=-\inf_{z \in E} (\Theta(z) + \Xi(z))\\
&= -\max_{\gamma \in E^*} (-\Theta^*(-\gamma) - \Xi^*(\gamma)) \\
&= \min_{\gamma \in E^*} (\Theta^*(-\gamma) + \Xi^*(\gamma)) \\
&= \min_{{\weightvect}\in \weightvectset,\ \gamma\in \Pi(\mu, \nu_{\weightvect})} \(\int_{X\times Y} c \dga + F( \weightvect) \)
\end{align*}
proving the claimed strong duality.
\end{proof}
\subsection{Existence of dual maximizers}
We will now show the existence of maximizers for the dual problem \eqref{eqn: dual problem}. It is convenient at this point to introduce the notion of $c$ and $c^*$-transforms, and $c$-convexity. Note carefully, since we are in the semi-discrete case the $c^*$-transform of a function defined on $X$ will be a vector in $\R^N$, while the $c$-transform of a vector in $\R^N$ will be a function whose domain is $X$.
\begin{defin}\label{def: c-transforms}
	If $\varphi: X\to \R\cup \{+\infty\}$ (which is not identically $+\infty$) and $\psi\in \R^N$, their $c$- and $c^*$-transforms are a vector $\varphi^c\in \R^N$ and a function $\psi^{c^*}: X\to \R\cup \{+\infty\}$ respectively, defined by
	\begin{align*}
	(\varphi^c)^j:=\sup_{x\in X}(-c(x, y_j)-\varphi(x)),\quad (\psi^{c^*})(x):=\max_{1\leq j\leq N}(-c(x, y_j)-\psi^j).
	\end{align*}
	
	If $\varphi: X\to \R\cup \{+\infty\}$ is the $c^*$-transform of some vector in $\R^N$, we say \emph{$\varphi$ is a $c$-convex function}. 
	We say a pair $(\varphi, \psi)$ with $\varphi: X\to \R\cup \{+\infty\}$ and $\psi\in R^N$ is a \emph{$c$-conjugate pair} if $\varphi=\psi^{c^*}$ and $\psi=\psi^{c^*c}$.
\end{defin}

Note just from the definition, if $-\varphi-\psi\leq c$, then $-\varphi(x)\leq -\psi^{c^*}(x)$ and $-\psi^j\leq -(\varphi^c)^j$ for all $x\in X$ and $1\leq j\leq N$, while $-\varphi-(\varphi^c)\leq c$, $-(\psi^{c^*})-\psi\leq c$ always holds.

As in the classical optimal transport case (see, for example \cite[Proposition 1.11]{Santambrogio15}), we utilize the $c$- and $c^*$-transforms of functions to obtain compactness.
\begin{prop}\label{prop: dual max exists}
There exists at least one maximizer of the dual problem \eqref{eqn: dual problem} that is a $c$-conjugate pair. 
Moreover, if $(\varphi, \psi)$ is any maximizing pair in \eqref{eqn: dual problem}, then it must be that $\varphi\equiv \psi^{c^*}$ on $X$.
\end{prop}
\begin{proof}
Let $(\varphi_n, \psi_n)$ be an admissible, maximizing sequence for \eqref{eqn: dual problem}. We may assume $\varphi_n=\psi_n^{c^*}$ for this sequence as $-\int_X\varphi_nd\mu\leq -\int_X\psi_n^{c^*}d\mu$, and $\psi_n=\psi_n^{c^*c}$ as 
\begin{align*}
-F^*(\psi_n)=\inf_{\weightvect\in \dom(F)}(\inner{\weightvect}{-\psi_n}+F(\weightvect))\leq \inf_{\weightvect\in \dom(F)}(\inner{\weightvect}{-(\psi_n^{c^*c})}+F(\weightvect))=-F^*(\psi_n^{c^*c}),
\end{align*}
using the fact that $\weightvect^j\geq 0$ for all $\weightvect\in \dom(F)$ and $-\psi_n\leq -\psi_n^{c^*c}$ componentwise. 
Since $(\psi_n+r\onevect)^{c^*}(x)=\psi_n^{c^*}(x)-r$ for any $r$, the above along with \eqref{eqn: invariance of dual} implies that replacing $\psi_n$ by $(\psi_n+r\onevect)^{c^*c}$ and taking $\varphi_n=(\psi_n+r\onevect)^{c^*}$ does not change the values of $-\int_X\varphi_nd\mu-F^*(\psi_n)$. Hence just as in the proof of \cite[Proposition 1.11]{Santambrogio15} there exists a subsequence, that we do not relabel, of $(\varphi_n, \psi_n)$ that converges ($\varphi_n$ uniformly on $X$ and $\psi_n$ in $\R^N$) to some $(\varphi, \psi)$. Since $-F^*$ is a concave function, finite on all of $\R^N$ by compactness of $\dom(F)$, it is continuous on $\R^N$, hence we obtain that $(\varphi, \psi)$ is a maximizer in \eqref{eqn: dual problem}. We can replace the pair by $(\psi^{c^*}, \psi^{c^*c})$ which only increases the value of the associated functional, hence there exists at least one $c$-conjugate maximizing pair.

	Finally let $(\varphi, \psi)$ be a maximizing pair. Recall $-\varphi\leq -\psi^{c^*}$ on $X$, if the inequality is strict anywhere, by continuity, strict inequality holds on a neighborhood relatively open in $X$. Since $X = \spt \mu$ we would have $-\int_X\varphi\dmu<-\int_X\psi^{c^*}\dmu$, contradicting that $(\varphi, \psi)$ is a maximizing pair. Thus we must have $\varphi\equiv \psi^{c^*}$.
\end{proof}

\section{Relationships between dual and primal optimizers} \label{section: dual characterization}
In this section we first show various properties of maximizers of the dual problem \eqref{eqn: dual problem}, followed by relationships between these maximizers and minimizers of the primal problem \eqref{eqn: primal problem}. As a consequence, we will obtain a way to characterize optimizers in both problems, along with uniqueness of minimizers under some mild conditions.

\subsection{Dual maximizers from primal minimizers}
In this subsection, we start with a minimizer in the primal problem and show how it relates to maximizers of the dual problem.
\begin{prop}\label{prop: primal v are dual max}
 Suppose $(\ti \gamma, \ti \weightvect)\in  \Pi(\mu, \nu_{\ti \weightvect})\times \weightvectset$ are a minimizing pair in the primal problem \eqref{eqn: primal problem}. Then for any $(\hat\varphi, \hat \psi)\in C(X)\times \R^N$ which are maximizers in the dual problem \eqref{eqn: dual problem}, we must have 
\begin{align*}
 -F^*(\hat\psi)=-\inner{\tilde \weightvect}{\hat \psi}+F(\tilde \weightvect).
\end{align*}
Additionally if $\tilde \weightvect^j>0$ for some $1\leq j\leq N$, we must have $\(\hat\psi^{c^*c}\)^j=\hat \psi^j$ for that index $j$.
\end{prop}

\begin{proof}
Let $(\ti \gamma, \ti \weightvect)$ be a minimizing pair in the primal problem \eqref{eqn: primal problem} and  $(\hat\varphi, \hat \psi)\in C(X)\times \R^N$ be a maximizing pair in the dual problem \eqref{eqn: dual problem}. By definition of the Legendre transform, we have
\begin{align*}
 -F^*(\hat\psi)=\inf_{\weightvect\in \weightvectset}(-\inner{\weightvect}{ \hat\psi} + F(\weightvect))\leq-\inner{\tilde \weightvect}{\hat \psi}+F(\tilde \weightvect).
\end{align*}

For the opposite inequality, first by Theorem \ref{thm: strong duality} we have
\begin{align}
\int_{X \times Y} c d\ti \gamma + F(\ti \weightvect) &=\sup_{-\varphi-\psi\leq c} \(-\int_X \varphi \dmu -F^*(\psi)\)\notag\\
&= \sup_{-\varphi-\psi\leq c} \(-\int_X \varphi \dmu + \inf_{\weightvect\in \weightvectset}(-\inner{\weightvect}{\psi} + F(\weightvect))\)\notag \\
&\leq \sup_{-\varphi-\psi\leq c} \(-\int_X \varphi \dmu -\inner{\ti \weightvect}{\psi}\) + F(\ti \weightvect).\label{eqn: primal to dual 1}
\end{align}
As from the discussion in the introduction, $\tilde \gamma$ is a minimizer in the classical Kantorovich problem \eqref{eqn: kantorovich} with the right marginal equal to $\nu_{\tilde{\weightvect}}$. Thus by Kantorovich duality in the classical optimal transport problem (\cite[Theorem 1.3]{Villani03}) we have $\sup_{-\varphi-\psi\leq c} \(-\int_X \varphi \dmu - \inner{\ti \weightvect }{ \psi}\)=\int_{X \times Y} c d\ti \gamma$.
%
%
In particular, the inequality in the middle of the calculation leading to \eqref{eqn: primal to dual 1} above is an equality, thus we have
\begin{align*}
- \int_X\hat \varphi \dmu-F^*(\hat \psi)&=\sup_{-\varphi-\psi\leq c} \(-\int_X \varphi \dmu -F^*(\psi)\)=\sup_{-\varphi-\psi\leq c} \(-\int_X \varphi \dmu -\inner{ \ti \weightvect }{ \psi} + F(\ti \weightvect)\)\\
 &\geq -\int_X\hat\varphi\dmu-\inner{\ti \weightvect}{ \hat \psi}+F(\ti \weightvect),
\end{align*}
finishing the first claim of the proof.

Now suppose $\tilde \weightvect^{ j}>0$ for some $1\leq  j\leq N$. Recall by Proposition \ref{prop: dual max exists}, we must have $\hat \varphi\equiv \hat\psi^{c^*}$, and we also have $\hat\psi^k\leq (\hat\psi^{c^*c})^k$ for all $1\leq k\leq N$; suppose by contradiction there is a strict inequality for the index $j$. Since $(\hat\psi^{c^*}, \hat\psi^{c^*c})$ is also a maximizer in \eqref{eqn: dual problem} we would then obtain
\begin{align*}
 -F^*(\hat\psi^{c^*c})&=-\inner{\ti \weightvect}{ \hat\psi^{c^*c}}+F(\ti \weightvect)=-\sum_{k\neq j}(\hat\psi^{c^*c})^k\ti\weightvect^k-(\hat\psi^{c^*c})^j\ti\weightvect^j+F(\ti \weightvect)\\
 &>-\sum_{k\neq j}(\hat\psi)^k\ti\weightvect^k-(\hat\psi)^j\ti\weightvect^j+F(\ti \weightvect)=-\inner{\ti \weightvect}{ \hat \psi}+F(\ti \weightvect)=-F^*(\hat \psi).
\end{align*}
However, this contradicts that $(\hat \varphi, \hat \psi)$ is a maximizer, thus we must have $\hat\psi= (\hat\psi^{c^*c})$.
\end{proof}

\subsection{Primal minimizers from dual maximizers}\label{subsection: primal from dual}
Next we aim to start with a maximizer in the dual problem and obtain a minimizer in the primal problem. In order to do so, we will have to add standard conditions under which a solution to the classical Kantorovich problem \eqref{eqn: kantorovich} can actually be written as solutions to the Monge problem \eqref{eqn: monge}. From this point on, we assume that for each $1\leq j\leq N$, the function $c(\cdot, y_j)\in C^1(X)$.
\begin{defin}\label{def: twist}
We say that the cost function $c$ satisfies the \emph{twist} condition if for each $x_0\in X$ and $j\neq k$, we have
\begin{align}\label{eqn: twist}
-\nabla_x c(x_0, y_j)\neq -\nabla_x c(x_0, y_k). 
\end{align}
\end{defin}
\begin{rmk}\label{rmk: brenier theorem}
By the generalized Brenier's theorem \cite[Theorem 10.28]{Villani09}, if $c$ satisfies the twist condition and $\mu$ is absolutely continuous with respect to Lebesgue measure, any solution of the Kantorovich problem \eqref{eqn: kantorovich} can be written in the form $(\Id\times T)_\#\mu$ where $T$ is a mapping defined $\mu$-a.e. that is a solution to the Monge problem \eqref{eqn: monge}. In particular, under these conditions, a solution $\gamma$ of \eqref{eqn: kantorovich} must be supported on the graph of a mapping from $X$ to $Y$ that is single valued $\mu$-a.e..

Now let $\psi\in \R^N$ and define $\weightvect\in \R^N$ by $\weightvect^j=\mu\(\{x\in X\mid  \psi^{c^*}(x)=-c(x, y_j)-\psi^j\}\)$ for each $j$. When $c$ satisfies the twist condition and $\mu$ is absolutely continuous with respect to Lebesgue measure, it can be seen by the implicit function theorem that $\weightvect\in \weightvectset$, and for $\mu$-a.e. $x$ there is a unique index $j$ such that $\psi^{c^*}(x)=-c(x, y_j)-\psi^j$; define $T_\psi: X\to Y$ by $T_\psi(x)=y_j$ whenever $j$ is the unique index associated to $x$. It is clear that $(T_\psi)_\#\mu=\nu_\weightvect$, hence by \cite[Remark 5.13]{Villani09}, we can see that $\gamma_\weightvect:=(\Id \times T_\psi)_\#\mu$ is a solution to the Kantorovich problem \eqref{eqn: kantorovich} with $\nu=\nu_\weightvect$.
%
\end{rmk}
\begin{prop}\label{prop: primal from dual}
Suppose $c$ satisfies the twist condition \eqref{eqn: twist}, and $\mu$ is absolutely continuous with respect to Lebesgue measure. Let $(\conj \varphi, \conj \psi)$ be a maximizing pair in the dual problem and define $\ti \weightvect\in \weightvectset$ by 
\begin{align*}
\ti \weightvect^j:=\mu\(\{x\in X\mid \conj \varphi(x)=-c(x, y_j)-\conj\psi^j\}\),
\end{align*}
and take $\ti \gamma\in \Pi(\mu, \nu_{\ti \weightvect})$ to be the solution of the classical Kantorovich problem \eqref{eqn: kantorovich} with $\nu=\nu_{\ti \weightvect}$. Then $(\ti \gamma, \ti \weightvect)$ is a minimizing pair in the primal problem \eqref{eqn: primal problem}
\end{prop}
\begin{proof}
Let $(\conj \varphi, \conj \psi)$ be a maximizing pair in the dual problem. By Proposition \ref{prop: dual max exists} we see that $\conj \varphi\equiv \conj \psi^{c^*}$, and we easily see that replacing $\conj \psi$ with $\conj \psi^{c^*c}$ does not change the vector $\ti \weightvect$, so we make this replacement.

Since $-\conj\varphi(x)-\conj\psi^j\leq c(x, y_j)$ for all $x$, $j$, by Kantorovich duality in the classical optimal transport problem (\cite[Theorem 1.3]{Villani03}), we have for any fixed $\weightvect\in \weightvectset$ that 
\begin{align*}
\ell(\weightvect):= -\int_X \conj\varphi \dmu-\inner{\weightvect}{\conj\psi}\leq\min_{\gamma \in \Pi(\mu, \nu_{\weightvect})} \int_{X \times Y} c \dga=:\mathcal{C}(\weightvect).
\end{align*}
At the same time by strong duality, Theorem \ref{thm: strong duality},
\begin{align*}
   \inf_{\weightvect\in \weightvectset}[F(\weightvect)+\ell(\weightvect)]&=   -\int_X \conj\varphi \dmu+\inf_{\weightvect\in \weightvectset}[F(\weightvect)-\inner{\weightvect}{\conj\psi}]=-\int_X \conj\varphi \dmu - F^*(\conj\psi)\\
&= \min_{\weightvect\in \weightvectset,\ \gamma\in \Pi(\mu, \nu_\weightvect)} (\int_{X \times Y} c \dga + F(\weightvect))
=\min_{\weightvect\in \weightvectset}[F(\weightvect)+\mathcal{C}(\weightvect) ].
\end{align*}
Thus we obtain that $F+\ell\leq F+\mathcal{C}$ pointwise everywhere on $\weightvectset$, and the above calculation shows that $F+\ell$ attains its minimum value over $\weightvectset$, at the same point as $F+\mathcal{C}$; say this point is $\weightvect_{\min}$.

We now claim that $\mathcal{C}$ is strictly convex on $\weightvectset$. Let $\weightvect_1$, $\weightvect_2\in \weightvectset$, $t\in [0, 1]$, and let $\gamma_{\weightvect_1}\in \Pi(\mu, \nu_{\weightvect_1})$, $\gamma_{\weightvect_2}\in \Pi(\mu, \nu_{\weightvect_2})$ be optimal in the minimum defining $\mathcal{C}(\weightvect_1)$ and $\mathcal{C}(\weightvect_2)$ respectively. Then 
\begin{align*}
 \mathcal{C}((1-t)\weightvect_1+t\weightvect_2)&\leq \int_{X\times Y} cd((1-t)\gamma_{\weightvect_1}+t\gamma_{\weightvect_2})=(1-t)\mathcal{C}(\weightvect_1)+t\mathcal{C}(\weightvect_2)
\end{align*}
which shows that $\mathcal{C}$ is convex on $\weightvectset$. Now suppose that we had equality in the above expression for some $t \in (0,1)$. This would mean that the measure achieving the minimum in $\mathcal{C}(t\weightvect_1 + (1-t)\weightvect_2)$ is $\gamma_t:=t\gamma_{\weightvect_1} + (1-t)\gamma_{\weightvect_2}$. By Remark \ref{rmk: brenier theorem}, we see $\gamma_{\weightvect_1}=(\Id \times T_{\weightvect_1})_{\#} \mu$ and $\gamma_{\weightvect_2}=(\Id \times T_{\weightvect_2})_{\#} \mu$ for some mappings $T_{\weightvect_1}$ and $T_{\weightvect_2}: X\to Y$, which are single valued $\mu$-a.e.. However, if $\gamma_{\weightvect_1}\neq \gamma_{\weightvect_2}$, this would imply $T_{\weightvect_1}\neq T_{\weightvect_2}$ on a set of nonzero $\mu$-measure. Clearly we must have that $\gamma_t$ is supported on the union of the graphs of $T_{\weightvect_1}$ and $T_{\weightvect_2}$, this leads to a contradiction since again by Remark \ref{rmk: brenier theorem}, $\gamma_t$ must be supported on the graph of a $\mu$-a.e. single valued mapping. Therefore $\mathcal{C}$ must actually be strictly convex on $\weightvectset$.

By Remark \ref{rmk: brenier theorem} and the choice of $\ti \weightvect$, we have $\mathcal{C}(\ti \weightvect)=\ell(\ti \weightvect)$, hence for any $t\in [0, 1]$ we must have 
\begin{align*}
\ell((1-t)\weightvect_{\min}+t\ti \weightvect)&\leq \mathcal{C}((1-t)\weightvect_{\min}+t\ti \weightvect)\leq (1-t)\mathcal{C}(\weightvect_{\min})+t\mathcal{C}(\ti \weightvect)\\
 &=(1-t)\ell(\weightvect_{\min})+t\ell(\ti \weightvect)=\ell((1-t)\weightvect_{\min}+t\ti \weightvect),
\end{align*}
i.e., $\ell\equiv \mathcal{C}$ on the segment $[\weightvect_{\min}, \ti \weightvect]$.

However recall that $\mathcal{C}$ is strictly convex. The only way for a strictly convex to equal a affine function on $[\weightvect_{\min}, \ti \weightvect]$ is if $\weightvect_{\min} = \ti \weightvect$. It is then clear that $(\ti \gamma, \ti \weightvect)$ is a minimizer in the primal problem \eqref{eqn: primal problem}.  
\end{proof}
The above proof also immediately yields the following corollary.
\begin{cor}\label{cor: unique minimizer}
 If  $\mu$ is absolutely continuous and $c$ satisfies the twist condition \eqref{eqn: twist}, minimizers in the primal problem \eqref{eqn: primal problem} are unique.
\end{cor}

\subsection{Characterization of optimizers}\label{subsec: characterization}
Using the above properties of dual and primal optimizers, we obtain a characterization for optimizers in both problems.
\begin{defin}\label{def: subdifferential}
 If $F: \R^n\to \R\cup \{+\infty\}$ is a convex function, its \emph{subdifferential} at a point $x\in \R^n$ is defined as
\begin{align*}
 \subdiff{F}{x}:=\{p\in \R^n\mid F(y)\geq F(x)+\inner{y-x}{p},\forall y\in \R^n\}
\end{align*}
\end{defin}

\begin{thm}\label{thm: conditions for optimizers}
Assume $\mu$ is absolutely continuous with respect to Lebesgue measure and $c$ satisfies the twist condition \eqref{eqn: twist}.

If $(\hat \phi, \hat \psi)$ is a maximizing pair in the dual problem \eqref{eqn: dual problem} and $(\ti \gamma, \ti \weightvect)$ is a minimizer in the primal problem \eqref{eqn: primal problem}, $\hat \psi$ and $\ti \weightvect$ satisfy the conditions (i) and (ii) below, 
\begin{enumerate}[{(i)}]
 \item $ \hat\psi \in \partial F(\ti \weightvect)$
  \item $\ti \weightvect^j = \mu\(\{x\in X \mid -c(x,y_j) - \hat \psi^j =  \hat\psi^{c^*}(x)\}\)$.
\end{enumerate}
Furthermore, if $\ti \weightvect^j>0$ for some $1\leq j\leq N$, we have
\begin{enumerate}[{(iii)}]
 \item $\hat \psi^j = \(\hat \psi^{c^*c}\)^j$.
\end{enumerate}
Conversely, if $\ti\weightvect\in \weightvectset$ and $\hat \psi \in \R^N$ are such that conditions (i) and (ii) above hold, then defining $T_{\hat \psi}$  as in Remark \ref{rmk: brenier theorem}, the pairs $(\hat \psi^{c^*}, \hat \psi)$ and $((\Id\times T_{\hat \psi})_\#\mu, \ti\weightvect)$ are maximizing and minimizing pairs in the dual and primal problem respectively.
\end{thm}

\begin{proof}
The first claims above follow immediately from Proposition \ref{prop: primal v are dual max} combined with \cite[Theorem 23.5]{Rockafellar70}, and Proposition \ref{prop: primal from dual} combined with the uniqueness of minimizers from Corollary \ref{cor: unique minimizer}.

Now suppose $\hat \psi\in \R^N$, $\ti \weightvect\in \weightvectset$ satisfy (i) and (ii). We have
\begin{align}
\sup_{-\varphi-\psi\leq c} \(-\int_X \varphi \dmu - F^*(\psi)\)
&\geq -\int_X \hat \psi^{c^*} \dmu - F^*(\hat \psi)= -\int_X \hat \psi^{c^*} \dmu - \inner{\ti \weightvect }{ \hat \psi} + F(\ti \weightvect)\label{eqn: characterization 1}
\end{align}
where this last equality is from condition (i) and \cite[Theorem 23.5]{Rockafellar70} again. Let $T_{\hat \psi}$ be defined as in Remark \ref{rmk: brenier theorem}, by condition (ii), we see that $\ti \gamma:=(\Id\times T_{\hat \psi})_\#\mu$ is a minimizer in the classical Kantorovich problem \eqref{eqn: kantorovich} with $\nu=\nu_{\tilde \weightvect}$. Let $x\in X$ be such that $T(x)$ is well-defined. By definition, this means that  $\hat \psi^{c^*}(x)+\hat \psi^j=-c(x, y_j)$ where $y_j=T(x)$. Since (see Remark \ref{rmk: brenier theorem}) the set of such $x$ has full $\mu$ measure, the union of $(x, T(x))$ over such $x$ has full $\ti \gamma$ measure. Thus by \cite[Theorem 5.10 and Remark 5.13]{Villani09}, we have that $(\hat \psi^{c^*}, \hat \psi)$ is a maximizer in the classical Kantorovich dual problem, and in particular $-\int_X \hat \psi^{c^*} \dmu - \inner{\ti \weightvect}{\hat \psi} = \inf_{\gamma\in \Pi(\mu, \nu_{\tilde{\weightvect}})}\( \int_{X \times Y} c \dga\)$. Thus we can calculate,
\[
-\int_X \hat \psi^{c^*} \dmu - \inner{\ti \weightvect}{\hat \psi} + F(\ti \weightvect)
&= \inf_{\gamma\in \Pi(\mu, \nu_{\tilde{\weightvect}})}\( \int_{X \times Y} c \dga\) + F(\ti \weightvect) \geq \inf_{\weightvect\in \weightvectset,\ \gamma\in\Pi(\mu, \nu_\weightvect)} \(\int_{X \times Y} c \dga + F(\weightvect)\)  \\
&= \sup_{-\varphi-\psi\leq c} \(-\int_X \varphi \dmu - F^*(\psi)\)
\]
with this last equality from by Theorem \ref{thm: strong duality}.
Combining this with \eqref{eqn: characterization 1}, we see
\[
\int_{X \times Y} c d \ti \gamma + F(\ti \weightvect) =\inf_{\gamma\in \Pi(\mu, \nu_{\tilde{\weightvect}})}\( \int_{X \times Y} c \dga\) + F(\ti \weightvect) 
= \inf_{\weightvect\in \weightvectset,\ \gamma\in\Pi(\mu, \nu_\weightvect)} \(\int_{X \times Y} c \dga + F(\weightvect)\)
\]
 hence $(\ti \gamma, \ti \weightvect)$ is a minimizing pair in the primal problem. The above calculations also yield

\[
\sup_{-\varphi-\psi\leq c} \(-\int_X \varphi \dmu - F^*(\psi)\) 
=- \int_X \hat \psi^{c^*} \dmu - F^*(\hat \psi), 
\]
thus $(\hat \psi^{c^*}, \hat \psi)$ is a maximizing pair in the dual problem. 
\end{proof}

\section{Stability of $F$}\label{section: approximation}
In this section we show the stability of minimizers to our primal problem \eqref{eqn: primal problem}, under perturbations of the storage fee function $F$. First we estimate the change in the minimum value of the problem.
\begin{prop}\label{eqn: bound on minimum value difference}
Let $F_1$ and $F_2: \R^N\to \R\cup \{+\infty\}$ be lower semi-continuous and proper, and write $\mathfrak{m}_{F_i}$ for the minimum value attained in \eqref{eqn: monge ver} with some fixed measure $\mu$ and the choice $F=F_i$, $i=1$ or $i=2$. Then
 \[
\abs{\mathfrak{m}_{F_1}-\mathfrak{m}_{F_2}}\leq \norm{F_1 - F_2}_{L^\infty(\dom(F_1)\cup \dom(F_2))}
\]
\end{prop}

\begin{proof}
Let the pair $(\ti \gamma_2, \ti \weightvect_2)$ achieve the minimum value in $\mathfrak{m}_{F_2}$ (in particular, $F_2(\ti \weightvect_2)$ is finite). If $F_1(\ti \weightvect_2)=+\infty$, we would have $\norm{F_1-F_2}_{L^\infty(\dom(F_1)\cup \dom(F_2))} = +\infty$ hence the claim is trivial. Thus we may assume $F_1(\ti \weightvect_2)$ is finite, then,
\begin{align*}
 \mathfrak{m}_{F_1}-\mathfrak{m}_{F_2}&\leq \(\(\int_{X \times Y} c d\ti \gamma_2 + F_1(\ti \weightvect_2)\) - \(\int_{X \times Y} c d\ti \gamma_2 + F_2(\ti \weightvect_2)\)\)\\
 &=F_1(\ti \weightvect_2)-F_2(\ti \weightvect_2)\leq \norm{F_1-F_2}_{L^\infty(\dom(F_1)\cup \dom(F_2))}.
\end{align*}
The same argument reversing the roles of $F_1$ and $F_2$ finish the proof.
%
%
%
%
%
%
%
%
%
\end{proof}

The above shows that if $F_i$ converges to $F$ uniformly, then $\mathfrak{m}_{F_i}$ converges to $\mathfrak{m}_F$. Next we will prove that the minimizing plans weakly converge to a minimizer of the original problem.

\begin{thm}
Let $(\ti \gamma, \ti \weightvect)$ minimize $\mathcal{C}_F$ and $(\ti \gamma_j, \ti \weightvect_j)$ minimize $\mathcal{C}_{F_j}$ for each $j$, where $F$, $F_j$ are all proper, convex functions with compact essential domains contained in $\weightvectset$. If $$\lim_{j\to\infty}\norm{F_j - F}_{L^\infty(\dom(F_j)\cup \dom(F))}=0,$$ then $\ti\weightvect_j$ converges to $\ti\weightvect$, and $\ti \gamma_j$ converges weakly to $\ti \gamma$.
\end{thm}

\begin{proof}
Let $\mathcal{C}_F(\weightvect) = \inf_{ \gamma\in \Pi(\mu, \nu_\weightvect)} \int_{X \times Y} c \dga + F(\weightvect)$ and $\mathcal{C}_{F_j}(\weightvect)$ defined analogously. By the proof of Proposition \ref{prop: primal from dual} and Corollary \ref{cor: unique minimizer}, we see that $\mathcal{C}_F$ and $\mathcal{C}_{F_j}$ are convex functions on $\weightvectset$ each of which have unique minimizers, given by $\ti\gamma$ and $\ti\gamma_j$ respectively. Then by Lemma \ref{eqn: bound on minimum value difference},
\begin{align*}
 \abs{\mathcal{C}_{F}(\ti\weightvect_j)-\mathcal{C}_F(\ti\weightvect)}&\leq  \abs{\mathcal{C}_{F_j}(\ti\weightvect_j)-\mathcal{C}_F(\ti\weightvect)}+ \abs{\mathcal{C}_{F_j}(\ti\weightvect_j)-\mathcal{C}_F(\ti\weightvect_j)}=\abs{\mathfrak{m}_{F_j}-\mathfrak{m}_{F}}+\abs{F_j(\ti\weightvect_j)-F(\ti\weightvect_j)}\\
 &\leq 2\norm{F_j - F}_{L^\infty(\dom(F_j)\cup \dom(F))}\to 0,\ j\to\infty.
\end{align*}
By compactness of $\weightvectset$, any subsequence of $\{\ti\weightvect_j\}_{j=1}^\infty$ has a convergent subsequence, by the above calculation and strict convexity of $\mathcal{C}_F$ on $\weightvectset$ all of these subsequential limits must be $\ti\weightvect$, hence we must have $\lim_{j\to\infty} \ti\weightvect_j=\ti\weightvect$.

Now suppose by contradiction that $\ti\gamma_j$ does not converge weakly to $\ti\gamma$. Since $\Gamma_\mu$ is weakly compact by Lemma \ref{lem: tight}, we can extract a subsequence (which we do not relabel) which converges weakly to some limiting measure that is not $\ti\gamma$, say $\hat \gamma$. By the above paragraph combined with \eqref{eqn: right marginal converges} we have $ \ti\weightvect = \lim_{j\to\infty} \ti\weightvect_j= \hat \weightvect$ where $\hat \weightvect$ is such that the right marginal of $\hat \gamma$ is $\nu_{\hat \weightvect}$. We then have
\[
\int c d\hat \gamma + F(\hat \weightvect) 
&= \int c d\ti \gamma_j + F_j(\ti \lambda_j) + \(F(\hat \weightvect) - F(\ti \lambda_j)\)+\(F(\ti \lambda_j)-F_j(\ti \lambda_j)\) + \(\int c d\hat \gamma - \int c d\ti \gamma_j\) \\
&= \mathfrak{m}_{F_j} + \(F(\hat \lambda) - F(\ti \lambda_j)\) + \norm{F_j - F}_{L^\infty(\dom(F_j)\cup \dom(F))}+ \(\int c d\hat \gamma - \int c d\ti \gamma_j\). 
\]

Letting $j$ go to infinity we see that $\int c d\hat \gamma + F(\hat \weightvect) \leq \mathfrak{m}_{F}$ by Proposition \ref{eqn: bound on minimum value difference}, the lower semi-continuity of $F$, and the fact that $\ti \gamma_j$ converges weakly to $\hat \gamma$. Hence $\hat \gamma$ is a minimizer and by Corollary \ref{cor: unique minimizer} we see that $\hat \gamma = \ti \gamma$ as desired. 

%

\end{proof}

\bibliographystyle{alpha}
\bibliography{snowshoveling}

\end{document}